\documentclass[11pt,a4paper]{amsart}

\usepackage{amsmath}
\usepackage{amssymb}
\usepackage{enumerate}
\usepackage[final]{graphicx}

\usepackage[T1]{fontenc}
\usepackage{lmodern}
\usepackage{textcomp}
\usepackage{tikz}               

\usetikzlibrary{calc,through,backgrounds,arrows}
\usepackage{color}

\definecolor{webblue}{rgb}{0, 0, 1.0}  
\definecolor{webred}{rgb}{1.0, 0, 0}   

\newcommand{\PG}{\operatorname{PG}}

\usepackage[papersize={165mm,235mm},top=26mm,left=16mm,text={128mm,192mm},%
footskip=4pt,headsep=6mm,marginparsep=0pt]{geometry}
\usepackage{url}
\newtheorem{theorem}{Theorem}[section]

\newtheorem{lemma}[theorem]{Lemma}

\theoremstyle{definition}

\newtheorem{remark}[theorem]{Remark}

\numberwithin{equation}{section}

\title[A New Reducibility Result]{A New Reducibility Result for Minihypers in Finite Projective Geometries}

\author[I. Landjev]{Ivan Landjev}
\address[I. Landjev]{Institute of Mathematics and Informatics, BAS, 
	8,~Acad. G. Bonchev. str., 1113 Sofia, Bulgaria}
\email{\tt ivan@math.bas.bg}

\author[A. Rousseva]{Assia Rousseva}

\address[A. Rousseva]{Sofia University, Faculty of Mathematics and Informatics, 5,~James Bourchier Blvd, 1164 Sofia, Bulgaria}
\email{{\tt assia@fmi.uni-sofia.bg}}

\author[K. Vorob'ev]{Konstantin Vorobev}

\address[K. Vorob'ev]{Institute of Mathematics and Informatics, BAS,
	8,~Acad. G. Bonchev. str., 1113 Sofia, Bulgaria}
\email{{\tt konstantin.vorobev@gmail.com}}

\begin{document}
	
	\begin{abstract}
		In this paper we prove a new reducibility result for minihypers in projective geometries over finite fields. It is further used  to characterize
		the minihypers with parameters $(70,22)$ in $\PG(4,3)$. The latter  can be used to attack the existence problem for some hypothetical ternary Griesmer codes of dimension 6.\\
		
	\end{abstract}
	
	\maketitle

\section{Introduction}

In this note we present a reducibility theorem for minihypers in the projective geometries $\PG(r,q)$. It can be used to characterize the minihypers with parameters $(70,22)$ in $\PG(4,3)$. These are instrumental for solving the problem of the existence/nonexistence of several ternary Griesmer codes of dimension 6 \cite{Maruta-table,MO11,SM22,TOM08,YM09}.
We do not impose a restriction on the maximal point multiplicity although
for the 6-dimensional codes we need minihypers with a maximal point multiplicity of 2.
The paper is structured as follows. In section 2 we give some definitions and basic facts on arcs and minihypers in finite projective geometries. Section 3 contains our general reducibility result for minihypers. This theorem is then used in Section 4 to give a characterization of the $(70,22)$-minihypers in $\PG(4,3)$.

\section{Preliminaries}

In this section we introduce some basic notions and results on multisets of points in $\PG(r,q)$.

A \emph{multiset} in $\PG(r,q)$ is a mapping
$\mathcal{K}\colon\mathcal{P}\to\mathbb{N}_0$, from the pointset $\mathcal{P}$ 
of $\PG(r,q)$ to the non-negative integers.
For a subset $\mathcal{Q}$ of $\mathcal{P}$, we define 
$\mathcal{K}(\mathcal{Q})=\sum_{P\in\mathcal{Q}} \mathcal{K}(P)$.
The integer $\mathcal{K}(\mathcal{Q})$ is called the multiplicity of the 
subset $\mathcal{Q}$.
A point of multiplicity $i$ is called an $i$-point. Similarly, $i$-lines, $i$-planes,
$i$-solids are points, lines, 3-dimensional subspaces of multiplicity $i$. 
The integer $\mathcal{K}(\mathcal{P})$
is called the cardinality of the multiset $\mathcal{K}$.
Given a set of points $\mathcal{Q}\subseteq\mathcal{P}$ we define
\[\chi_{\mathcal{Q}}(P)=\left\{
	\begin{array}{ll}
1 & \text{ if } P\in\mathcal{Q}, \\
0 &	\text{ otherwise.}
\end{array}\right.\]
A multiset is called \emph{projective} if the multiplicity of each point takes on 
a value in $\{0,1\}$. A minihyper with
point multiplicities 0 and 1 is called a \emph{blocking set}.  
Minihypers were introduced by Hamada. Using this notion we indicate
the presence of multiple points.

A multiset $\mathcal{K}$ in $\PG(r,q)$ is called an
\emph{$(n,w)$-arc}, if: $(a)$ $\mathcal{K}(\mathcal{P}) = n$;
$(b)$ $\mathcal{K}(H)\le w$ for each hyperplane $H$ in $\PG(r,q)$, and
$(c)$ there is a hyperplane $H_0$ with $\mathcal{K}(H_0)=w$.
In a similar way, we define an \emph{$(n,w)$-minihyper} (or 
\emph{$(n,w)$-blocking set}) 
as a multiset $\mathcal{K}$ in $\PG(r,q)$ satisfying:
$(d)$ $\mathcal{K}(\mathcal{P}) = n$;
$(e)$ $\mathcal{K}(H)\ge w$ for each hyperplane $H$ in $\PG(k-1,q)$, and
$(f)$ there is a hyperplane $H_0$ with $\mathcal{K}(H_0)=w$.

For a multiset $\mathcal{K}$ in $\PG(k-1,q)$, 
we denote by $a_i$ the number of hyperplanes $H$ 
with $\mathcal{K}(H)=i$,
$i\ge0$. By $\Lambda_j$ we denote the number of points $P$ from
$\mathcal{P}$ with $\mathcal{K}(P)=j$.
The sequence $a_0,a_1,a_2,\ldots$ is called
\emph{the spectrum} of $\mathcal{K}$. Sometimes when we want to stress the fact
that a certain spectrum relates to the multiset $\mathcal{K}$, we
write $a_i(\mathcal{K})$, resp $\Lambda_j(\mathcal{K})$.

The existence of an $[n,k,d]_q$-code $C$ of full length
(no coordinate identically zero) is equivalent to that of a
$(n,n-d)$-arc in $\PG(k-1,q)$. From any generator matrix $G$
of $C$  one can define a multiset $\mathcal{K}$ with points
(with the corresponding multiplicities) the columns of $G$.
This correspondence between
$[n,k,d]_q$ codes and $(n,n-d)$-arcs maps isomorphic codes 
to projectively equivalent arcs and vice versa. If $\mathcal{K}$ is an
$(n,w)$-arc in $\PG(k-1,q)$ with maximal point multiplicity $s$, then 
$\mathcal{F}=s-\mathcal{K}$ is an $(sv_k-n,sv_{k-1}-w)$-minihyper.
Here as usual $v_k=(q^k-1)/(q-1)$.

Given an $(n,w)$-arc $\mathcal{K}$ in $\PG(k-1,q)$, 
we denote by $\gamma_i(\mathcal{K})$ the maximal
multiplicity of an $i$-dimensional flat in $\PG(k-1,q)$, i.e.
$\gamma_i(\mathcal{K}) = \max_{\delta}\mathcal{K}(\delta),\, 
i=0,\ldots,k-1$.
If $\mathcal{K}$ is
clear from the context we shall write just $\gamma_i$. 
It is well known that if $\mathcal{K}$ is an 
$(n,n-d)$-arc in $\PG(k-1,q)$ with $n=t+g_q(k,d)$ then
\[\gamma_j(\mathcal{K})\le t+\sum_{i=k-1-j}^{k-1} \left\lceil\frac{d}{q^{i}}\right\rceil.\]
In particular, for Griesmer arcs the maximal point multiplicity is at
most $\lceil d/q^{k-1}\rceil$.

In terms of minihypers we have the following lower bounds on the multiplicity of subspaces
of different dimensions. Both results are obtained by simple counting arguments.

\begin{lemma}
Let $\mathcal{F}$ be an $(n,w)$-minihyper in $\PG(r,q)$. Then
for any $s$-dimensional subspace $S$, it holds

\begin{equation}
	\label{eq:bound}
\mathcal{F}(S)\ge\left\lceil \frac{v_{r-s}w-v_{r-s-1}n}{q^{r-s-1}} \right\rceil.
\end{equation}
\end{lemma}

\begin{lemma}
	Let $\mathcal{F}$ be an $(n,w)$-minihyper in $\PG(r,q)$. Let $H$ be a hyperplane
	and let $T$ be a subspace of codimension 2 contained in $H$. Then
	
	\begin{equation}
	\label{eq:bound1}
	\mathcal{F}(T)\ge w-\frac{n-\mathcal{F}(H)}{q}.
	\end{equation}
\end{lemma}

The following argument will be used repeatedly throughout the paper.
Let $\mathcal{K}$ be a multiset in $\PG(r,q)$.
Fix an $i$-dimensional flat $\delta$ in
$\PG(r,q)$, with $\mathcal{K}({\delta})=t$.
Let further $\pi$ be a $j$-dimensional flat in $\PG(r,q)$ of complementary dimension,
i.e. $i+j=r-1$ and $\delta\cap\pi=\varnothing$.
Define the projection $\varphi=\varphi_{\delta,\pi}$ from $\delta$ onto $\pi$ by
\begin{equation}
\label{eq:project}
\varphi\colon
\left\{\begin{array}{lll}
\mathcal{P}\setminus\delta\ & \rightarrow\ & \pi \\
Q & \rightarrow & \pi\cap\langle\delta,Q\rangle .
\end{array}\right.
\end{equation}
In other words, every point $Q$ of $\PG(r,q)$, which is not in $\delta$,
is mapped in the point which is the
intersection of $\pi$ and the subspace generated by $\delta$ and $Q$.
As before, $\mathcal{P}$ denotes the set of points of $\PG(k-1,q)$.
Note that $\varphi$ maps $(i+s)$-flats containing
$\delta$ into $(s-1)$-flats in $\pi$.
Given a set of points $F\subset\pi$, define the induced multiset
$\mathcal{K}^{\varphi}$ by
\[\mathcal{K}^{\varphi}({F})\,=
\,\sum_{\varphi_{\delta,\pi}(P)\in\mathcal{F}} \mathcal{K}(P).\]
We shall exploit the obvious fact that if $S$ is a flat in $\PG(k-1,q)$
through $\delta$
then $\mathcal{K}^{\varphi}(\varphi(S))=\mathcal{K}(S)-t$.
In the next sections the subspace $\pi$ will be a plane.
A line in $\pi$ which is incident with the points $P_0,\ldots,P_q$
is called a line of type
$(\mathcal{K}^{\varphi}(P_0),\ldots,\mathcal{K}^{\varphi}(P_q))$.

The next few results have been  proved for linear codes, but can be easily
reformulated for arcs and blocking sets in finite projective geometries.
This is done in the next theorems.

\begin{theorem}
\label{thm:ward}\cite{HNW98}\
Let $\mathcal{K}$ be an $(n,w)$-arc (resp. $(n,w)$-minihyper) in $\PG(r,p)$,
where $p$ is a prime. Let further $w\equiv n\pmod{p^e}$ for some $e\ge1$.
Then for every hyperplane $H$ it holds that $\mathcal{K}(H)\equiv n\pmod{p^e}$.
\end{theorem}

An $(n,w)$-arc in $\PG(r,q)$ is called \emph{$t$-extendable} if the multiplicities of
some of the points can be increased by a total of $t$, so that
the obtained arc has parameters $(n+t,w)$. Similarly, an $(n,w)$-minihyper is
called \emph{$t$-reducible}  if the multiplicities of some of the points can be reduced by a total of $t$, so that the obtained multiset is an $(n-t,w)$-minihyper.
The following result by R. Hill and P. Lizak was proved initially for linear codes.

\begin{theorem}
\label{thm:hill}\cite{H99,HL95}\
Let $\mathcal{K}$ be an $(n,w)$-arc (resp. $(n,w)$-minihyper)
associated with a Griesmer code
in $\PG(r,q)$ with $(n-w,q)=1$, such that the multiplicities of all hyperplanes
are $n$ or $w$ modulo $q$. Then $\mathcal{K}$ is extendable to
an $(n+1,w)$-arc (resp. reducible to an $(n-1,w)$-minihyper).	
Moreover the point of extension (resp. the point of reduction) is uniquely determined.	
\end{theorem}

The next theorem is a more sophisticated extension result
by Hitoshi Kanda \cite{HK20} which applies only
to arcs (minihyper) in a geometry over $\mathbb{F}_3$.

\begin{theorem}
\label{thm:kanda}\cite{HK20}
Let $\mathcal{K}$ be an $(n,w)$-arc (resp. $(n,w)$-blocking set) in $\PG(r,3)$.
Assume further that the multiplicity of every hyperplane $H$ is
congruent to $n, n+1$, or $n+2$ modulo 9
(resp. $n-2, n-1$, or $n$ modulo 9 ). 
Then $\mathcal{K}$ is extendable to an $(n+2,w)$-arc
(resp. reducible to an $(n-2,w)$-minihyper).
\end{theorem}

\section{A Reducibility Theorem for Minihypers}

\begin{theorem}
\label{thm:main}
	Let $\mathcal{F}$ be an $(n,w)$-minihyper in $\PG(r,p)$, $p$ -- a prime, 
	with $w\equiv n-p\pmod{p^2}$ that has the following properties:
	
	\begin{enumerate}[(1)]
	\item for every hyperplane $H$ in $\PG(r,p)$ it holds
	$\mathcal{F}(H)\equiv n-p$ or $n\pmod{p^2}$;
	\item for every hyperplane $H$ with $\mathcal{F}(H)\equiv n-p\pmod{p^2}$, the restriction $\mathcal{F}\vert_H$ is reducible to a divisible minihyper with divisor $p$;
	\item for every hyperplane $H$ with $\mathcal{F}(H)\equiv n\pmod{p^2}$, the restriction $\mathcal{F}\vert_H$ is a divisible minihyper with divisor $p$.
	\end{enumerate}

	Then $\mathcal{F}=\mathcal{F}'+\chi_L$, where $\mathcal{F}'$ is a 
	$(n-v_2,w-v_1)$-minihyper and $L$ is a line. Moreover, the line $L$ is uniquely determined.
\end{theorem}

\begin{proof}
By (1) the multiplicities of the hyperplanes are $w+ip^2$ and $w+ip^2+p$, where 
$i=0,1,2,\ldots$. The hyperplanes $H$ of multiplicity $w+ip^2$ are reducible to
divisible minihypers by (2). Hence $\mathcal{F}\vert_H$ has parameters
$(w+ip^2,u_i)$, where $u_i\equiv w+ip^2-1\pmod{p}$.
Moreover, all hyperlines in $H$ have multiplicity $\equiv w-1$ or $w\pmod{p}$, or equivalently, $n-1$ or $n\pmod{p}$. The point of reduction is contained only in hyperplanes of multiplicity $\equiv n\pmod{p}$. Every hyperline that does not contain the point
of reduction is of multiplicity $\equiv n-1\pmod{p}$.

In the hyperplanes of multiplicity $w+p+ip^2$ all hyperlines have multiplicity
$\equiv n\pmod{p}$. So, for all hyperplanes $H$ through a hyperline of multiplicity 
$\equiv n-1\pmod{p}$,
one has $\mathcal{F}(H)\equiv n-p\pmod{p^2}$. 

Consider a hyperline $T$ of
multiplicity $\mathcal{F}(T)\equiv n\pmod{p}$. Denote by $x$ (resp. $y$) the number of hyperplanes  of multiplicity $\equiv n-p\pmod{p^2}$
(resp $\equiv n\pmod{p^2}$) through $p$. Obviously $x+y=p+1$.

Denote by $H_i$, $i=0,\ldots,p$,  the hyperplanes through $T$ and set 
$\mathcal{F}(T)\equiv n+\alpha p\pmod{p^2}$. Now
\begin{eqnarray*}
	n &=& \sum_i\mathcal{F}(H_i) -p\mathcal{F}(T) \\
	&\equiv & x(n-p)+yn -p(n+\alpha p) \pmod{p^2} \\
	&\equiv & n(x+y)-px - np - \alpha p^2 \pmod{p^2}.
\end{eqnarray*}
This implies $px\equiv 0\pmod{p^2}$, whence $x\equiv 0\pmod{p}$ and $y\equiv 1\pmod{p}$
(i.e. $y=1$ or $p+1$).

Define an arc $\widetilde{\mathcal{F}}$ in the dual geometry by
\[\widetilde{\mathcal{F}}(H)=\left\{
\begin{array}{cll}
1 & \text{ if } & \mathcal{F}(H)\equiv n\pmod{p^2}, \\
0 & \text{ if } & \mathcal{F}(H)\equiv n-p\pmod{p^2}.
\end{array}
\right.\]
By the fact proved above, if a line contains two 1-points with respect to
$\mathcal{F}$ then the whole line incident with them consists of 1-points. This means that all hyperplanes of multiplicity $n\pmod{p^2}$ form  a subspace in the dual geometry.

Consider a minimal hyperplane $H_0$, i.e. a hyperplane of multiplicity $w$. All hyperlines through the point of reduction are contained in a unique hyperplane of multiplicity $n\pmod{p^2}$. 	This implies that the number of the hyperplanes of multiplicty
$n\pmod{p^2}$ is equal to the number of the hyperlines in $H_0$ through a fixed point. This number is $v_{r-1}$. This implies that the hyperplanes of multiplicity $n\pmod{p^2}$ are all hyperplanes through a fixed line $L$.

It remains to show that all points on $L$ have multplicity at least 1 with respect to $\mathcal{F}$. Fix a minimal hyperplane $H_0$ and a hyperline $T$ in $H_0$ of multiplicity
$n-1\pmod{p}$. As noted above, all hyperplanes $H_i$, $i=0,\ldots,p$, through $T$ are also of multiplicity $n-p\pmod{p^2}$. Denote by $P_i$ the unique point of reducibility of the
minihyper $\mathcal{F}\vert_{H_i}$. All points $P_i$ are outside of the hyperlne $T$.
In addition, they are collinear since they form a blocking set with respect to the hyperplanes. Denote the line containing the points $P_i$ by $L'$.  
Assume there is a hyperplane $H$ of multiplicity $n\pmod{p^2}$ that meets $L'$ in a single point, since then it meets $H_1,\ldots,H_p$ in hyperlines of multiplicity
$n-1\pmod{p}$, which is impossible. Hence  every hyperplane of multiplicity $n\pmod{p^2}$ contains $L'$ and hence $L\equiv L'$.
\end{proof}

\section{The Classification of $(70,22)$-Minihypers in $\textrm{PG}(4,3)$}

As an application of Theorem~\ref{thm:main} 
we shall characterize the $(70,22)$-minihypers in $\PG(4,3)$.
This characterization is crucial for attacking the nonexistence of some ternary 6-dimensional codes whose existence is in doubt (cf \cite{Maruta-table}).

First, we claim without proof several characterization results for minihypers in $\PG(3,3)$. The proofs can be found  in \cite{LRR22-a}.

\begin{lemma}
	\label{lma:21-6}
	A $(21,6)$-minihyper in $\PG(3,3)$ is one of the following:		
	
	$(\alpha)$ the sum of a plane and two lines;
	
	$(\beta)$ a minihyper with one double point with $a_{12}=2$;
	
	$(\gamma)$ a projective minihyper and $a_{12}=1$.
\end{lemma}
\medskip

\begin{lemma}
	\label{lma:22-6}
	Every $(22,6)$-minihyper in $\PG(3,3)$ with maximal point multiplicity 2
	is either reducible to one of the
	$(21,6)$-minhypers, or else
	the sum of a projective plane of order 3 and a non-canonical
	planar $(9,2)$-minihyper (the complement of an oval in $\PG(2,3)$). 
\end{lemma}

The next lemma follows from the classification of the linear codes with parameters
$[50,4,33]_3$ and $[49,4,32]_3$ given in \cite{L98}.

\begin{lemma}
	\label{lma:30-9}
	A $(30,9)$-minihyper in $\PG(3,3)$ 	is one of the following:
	\begin{enumerate}[(a)]
		\item the sum of two planes and a line (a canonical minhyper);
		\item the union of two planes plus two skew lines meeting the two planes in their common line;
		\item the complement of a 10-cap in $\PG(3,3)$ (a projective minihyper).
	\end{enumerate}
	Every $(31,9)$-minihyper in $\PG(3,3)$ is reducible
	to a $(30,9)$-minihyper.
\end{lemma}

Now we state the theorem which is the main result in this section.
It describes the structure of the
$(70,22)$-minihypers in $\PG(4,3)$.

\begin{theorem}
	\label{thm:70-22}
	Let $\mathcal{F}$ be a $(70,22)$-minihyper in $\PG(4,3)$.
	Then $\mathcal{F}$ is one of the following:
	\begin{enumerate}[(A)]
		\item the sum of a solid and  a $(30,9)$-minihyper in $\PG(4,3)$;
		\item the sum of a $(66,21)$-minihyper in $\PG(4,3)$ and a line.
	\end{enumerate}	
\end{theorem}

\begin{remark}
The characterization of the $(66,21)$-minihypers
in $\PG(4,3)$ is given in \cite{LRR22}.	
\end{remark}

The proof of this theorem is split in several lemmas.
Until the end of the section, we shall assume that
$\mathcal{F}$ is a  $(70,22)$-minihyper in
$\PG(4,3)$. 

\begin{lemma}
\label{lma:divisible}
Let $\mathcal{F}$ be a $(70,22)$-minihyper in $\PG(4,3)$.
Then for
every solid $S$ in $\PG(4,3)$ it holds
$\mathcal {F}(S)\equiv 1\pmod{3}$ (i.e. $\mathcal{F}$ is a divisible minihyper).	
\end{lemma}

\begin{proof}
	Assume the maximal point multiplicity of such minihyper is $s$. Then
$s-\mathcal{F}$ is a $(121s-70,40s-22)$-arc in $\PG(4,3)$, which is associated with a $[121s-70,5,81s-48]_3$-code which is readily checked to be a Griesmer code. By Ward's Theorem this code is divisible and hence, in turn, $\mathcal{F}$ is also divisible. Thus for each solid $S$ in $\PG(4,3)$ $\mathcal{F}(S)\equiv1\pmod{3}$.
\end{proof}

\begin{lemma}
	\label{lma:49solid}
	Let $\mathcal{F}$ be a $(70,22)$-minihyper in $\PG(4,3)$. 
	If there exists a solid $S$ with $\mathcal{F}(S)\ge49$ then
	$\mathcal{F}$ is the sum of a  solid and a $(30,9)$ minihyper in $\PG(4,3)$.
\end{lemma}

\begin{proof}
Let us first note that a solid of multiplicity at least 49 does not have 0-points.
Otherwise, the cardinality of $\mathcal{F}$ is $|\mathcal{F}|\ge49+27=76$
since every line through the 0-point in that solid has to be blocked at least once.
Now obviously $\mathcal{F}-\chi_S$ is a $(30,9)$-minihyper since the multiplicity of each solid different from $S$ is reduced by 13.
\end{proof}

\begin{lemma}
\label{lma:28-46}
	Let $\mathcal{F}$ be a $(70,22)$-minihyper in $\PG(4,3)$ in which every solid $S$ is 
	of multiplicity $\mathcal{F}(S)<49$. Then there exist no solid $S$ with:
	\begin{enumerate}[(a)]
	\item 	$\mathcal{F}(S)=28$;
	\item 	$\mathcal{F}(S)=37$;
	\item 	$\mathcal{F}(S)=46$;	
	\end{enumerate}
\end{lemma}

\begin{proof}
(a) Let us start by noting that every $(28,8)$-minihyper in $\PG(3,3)$ is 
reducible to a $(26,8)$-minihyper. Hence it is the sum of two planes and two points.
The 22-solids through an 8-plane are just the irreducible $(22,6)$-minihypers
(the sum of a plane and the complement to an oval). There are three possibilities for such minihypers presented at the pictures below:

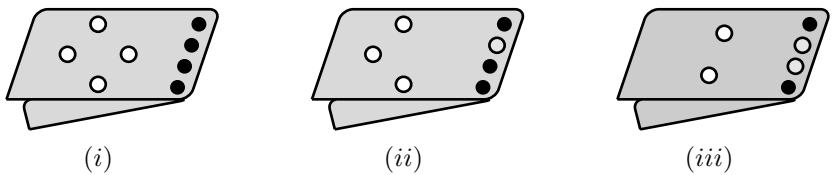
\begin{center}
\begin{tikzpicture}[line width=1pt, scale=0.4]
	\draw[black, rounded corners] (0,0)--(6,0)--(7,3)--(1,3)--(0,0) [fill=black!15];
	\draw[black, rounded corners] (0.75,-1)--(6,0)--(0.5,0)--(0.75,-1) [fill=black!15];

	\draw[black]  (5.6,0.4) circle (0.2cm) [fill=black];
	\draw[black]  (5.83,1.1) circle (0.2cm) [fill=black];
	\draw[black]  (6.06,1.8) circle (0.2cm) [fill=black];
	\draw[black]  (6.3,2.5) circle (0.2cm) [fill=black];

	\draw[black]  (3,0.5) circle (0.25cm) [fill=white];
	\draw[black]  (2,1.5) circle (0.25cm) [fill=white];
	\draw[black]  (4,1.5) circle (0.25cm) [fill=white];
	\draw[black]  (3,2.5) circle (0.25cm) [fill=white];

	\draw[black, rounded corners] (10,0)--(16,0)--(17,3)--(11,3)--(10,0) [fill=black!15];
	\draw[black, rounded corners] (10.75,-1)--(16,0)--(10.5,0)--(10.75,-1) [fill=black!15];

\draw[black]  (15.6,0.4) circle (0.2cm) [fill=black];
\draw[black]  (15.83,1.1) circle (0.2cm) [fill=black];
\draw[black]  (16.06,1.8) circle (0.25cm) [fill=black!10];
\draw[black]  (16.3,2.5) circle (0.2cm) [fill=black];

\draw[black]  (12,1.5) circle (0.25cm) [fill=white];
\draw[black]  (13,0.5) circle (0.25cm) [fill=white];
\draw[black]  (13,2.5) circle (0.25cm) [fill=white];

	\draw[black, rounded corners] (20,0)--(26,0)--(27,3)--(21,3)--(20,0) [fill=black!20];
	\draw[black, rounded corners] (20.75,-1)--(26,0)--(20.5,0)--(20.75,-1) [fill=black!20];

\draw[black]  (25.6,0.4) circle (0.2cm) [fill=black];
\draw[black]  (25.83,1.1) circle (0.25cm) [fill=black!10];
\draw[black]  (26.06,1.8) circle (0.25cm) [fill=black!10];
\draw[black]  (26.3,2.5) circle (0.2cm) [fill=black];

\draw[black]  (23,0.8) circle (0.25cm) [fill=white];
\draw[black]  (23.5,2.2) circle (0.25cm) [fill=white];

\draw (3,-2) node{\small{$(i)$}};
\draw (13,-2) node{\small{$(ii)$}};
\draw (23,-2) node{\small{$(iii)$}};

\end{tikzpicture}
\end{center}

In the picture,
the black points are 2-points, the white points are 0-points and the gray planes are planes of 1-points.

In case (i) the projection of the 22-plane from the 8-line is a line of type $(9,5,0,0)$
The projection of a 28-plane from the 8-line is of type
\[(18,0+\varepsilon_1,0+\varepsilon_2,0), \text{ or }
(9+\varepsilon_1,9+\varepsilon_2,0+\varepsilon_3,0),\]
with $\sum_i\varepsilon_i=2$ or $3$.
Now consider a 22-solid $S_0$ of type (i) and denote by $S_i$, $i=1,2,3$, the other three solids through the 8-plane $\pi$ consisting of four 2-points. 
  Now in the projection plane there exist three collinear $0$ or $0+\varepsilon$ points.  
The line incident with them is either of type $(18,0,0,0)$ or of type $(9+\varepsilon',0+\varepsilon'',0,0)$
which forces a solid of  multiplicity at most 19, a contradiction.

In case (ii) the proof is similar. We consider a projection from a 5-line in an 8-plane. 
Now the image of a 22-solid has one of the types
\[(11,3,2,1),\; (10,3,3,1), \; \text{ or } (10,3,2,2).\]
The image of a 28-solid is
\[(12+\varepsilon_1,3+\varepsilon_2,3+\varepsilon_3,3),\]
with $\sum_i\varepsilon_i=2$.
Now if the points of multiplicity at least 10 are not collinear then there is a line
in the projection plane of multiplicity at most 14, which forces a solid of multiplicity at most 19, a contradiction.
Otherwise the projection plane has a line of multiplicity at least 42.
This gives a solid with at least 47 points. This case was completed in Lemma~\ref{lma:49solid}.

(b) An 11-plane is forced to have a 2-line whence there are no $(37,11)$-minihypers.

(c) By Lemma~\ref{lma:22-6} a $(22,6)$-minihyper does not have 14-planes.
Fix a 14-plane $\pi$ in $S$ and denote by $S_0=S, S_1, S_2, S_3$ the solids throgh $\pi$
Clearly $\mathcal{F}(S_0)=46$, and $\mathcal{F}(S_i)\ge25$ for $i=1,2,3$. Then
\[|\mathcal{F}|=\sum_i\mathcal{F}(S_i)-3\mathcal{F}(\pi)\ge 46+3\cdot25-3\cdot14=79,\]
a contradiction.
\end{proof}

We have proved so far that if
$\mathcal{F}$ is a $(70,22)$-minihyper in $\PG(4,3)$ with maximal hyperplane of
multiplicity at most 46 then the possible multiplicities lie in the set 
$\{22,25,31,34,40,43\}$.

\begin{lemma}
	\label{lma:43}
	Let $\mathcal{F}$ be a $(70,22)$-minihyper in $\PG(4,3)$ and let $S$ solid of multiplicity 43. Then $\mathcal{F}\vert_S$ is a divisible minihyper
	with parameters $(43,13)$.
\end{lemma}

\begin{proof}
It is clear that in $S$ there are no planes of multiplicity 14 since 
$(22,6)$-minihypers do not have 14-planes. Furthermore, 15-planes in $S$ are also impossible. This in turn implies that there are no planes of
multiplicity $\equiv -1,0\pmod{3}$.
\end{proof}

\begin{lemma}
\label{lma:22irreducible}
Let $\mathcal{F}$ be a $(70,22)$-minihyper and let $S$ be a 22-solid.
Then $\mathcal{F}\vert_{S}$ is a reducible $(22,6)$-minihyper.
\end{lemma}

\begin{proof}
Denote by $(a_i)$ the spectrum of $\mathcal{F}$. Using simple counting argument,
we get the standard identities:
\begin{eqnarray*}
a_{22} + a_{25} + a_{31} + a_{34} + a_{40} + a_{43} &=& 121 \\
22a_{22} + 25a_{25} + 31a_{31} + 34a_{34} + 40a_{40} + 43a_{43} &=& 2800 \\
231a_{22} + 300a_{25} + 465a_{31} + 561a_{34} + 780a_{40} + 903a_{43} &=& 
35\cdot69\cdot13 + \\
       & & 27\sum_{i=2}^4 {i\choose2}\Lambda_i,
\end{eqnarray*}
where $\Lambda_i$ is the number of $i$-points.
(Note that every point is on a minimal hyperplane and the maximal point multiplicity
on a minimal hyperplane is 4.)
This implies
\begin{equation}
\label{eq:macw-main}
a_{31}+2a_{34}+5a_{40}+7a_{43}=10+\Lambda_2+3\Lambda_6+6\Lambda_4.
\end{equation}

The spectra $(b_i)$ of the irreducible $(22,6)$-minihypers are the following:

\begin{enumerate}[$(a)$]
\item $b_6=18$, $b_7=12$, $b_8=9$, $b_{22}=1$;
\item $b_6=18$, $b_7=12$, $b_8=8$, $b_{13}=1$, $b_{17}=1$;
\item $b_6=18$, $b_7=11$, $b_8=9$, $b_{13}=1$, $b_{16}=1$;
\item $b_6=17$, $b_7=12$, $b_8=9$, $b_{13}=1$, $b_{15}=1$.
\end{enumerate}

We are going to rule out each of the possibilities $(a)$--$(d)$.
The argument is similar in all four cases.
We fix a 22-solid $S_0$ and for each plane $\delta$ in $S_0$ we we consider the maximal contribution of the other three solids $S_1, S_2, S_3$, to the left-hand side of
(\ref{eq:macw-main}). The table below gives the maximal contributions
for planes $\delta$ in $S_0$ of different multiplicity;

\begin{center}
\begin{tabular}{c|ccc|c}
$\mathcal{F}(\delta)$ & $\mathcal{F}(S_1)$ & $\mathcal{F}(S_2)$ & $\mathcal{F}(S_3)$ &
contribution \\ \hline
6 & 22 & 22 & 22 &  0 \\
7 & 22 & 22 & 25 &  0 \\
8 & 22 & 25 & 25 &  0 \\
13 & 22 & 25 & 40 &  5 \\
15 & 22 & 31 & 40 &  6 \\
16 & 22 & 34 & 40 &  7 \\
17 & 25 & 34 & 40 &  7 \\
22 & 31 & 40 & 43 &  13 \\ \hline
\end{tabular}	
\end{center}

(a) The left-hand side is bounded from above by 13, since 6-, 7-, and 8-planes
give contribution of 0.
So, we have $13\ge 10+\Lambda_2$. But in this case we have obviously
$\Lambda_2\ge9$ (since the irreducible 22-plane alone has nine 2-points), which gives a contradiction.

(b) The left-hand side is at most $1\cdot5+1\cdot7=12\ge10+\Lambda_2$. In this case
$\Lambda_2\ge4$, a contradiction.

(c) We have again $1\cdot5+1\cdot7\ge10+\Lambda_2$, and $\Lambda_2\ge3$, a contradiction.

(d) The left-hand side is at most $1\cdot5+1\cdot6=11$. Hence $11\ge10+\Lambda_2$, but
$\Lambda_2\ge2$ again a contradiction.
\end{proof}

\begin{lemma}
	\label{lma:25-34}
	Let $\mathcal{F}$ be a $(70,22)$-minihyper in $\PG(4,3)$, and let $S$ solid of multiplicity $25$, or $34$. Then $\mathcal{F}\vert_S$ is a $(25,7)$-divisible minihyper
	in $\PG(3,3)$, respectively a divisible $(34,10)$-minihyper in $\PG(3,3)$.
\end{lemma}

\begin{proof}
Assume $\mathcal{F}\vert_S$ is a $(25,7)$-minihyper and assume that there exists an 8-plane in $S$. A minimal solid can have 8-planes only if it is irreducible. But such solids were ruled out by Lemma ~\ref{lma:22irreducible}. Hence counting the multiplicities of the solids through	$\pi$ we get $|\mathcal{F}|\ge 4\cdot25-3\cdot8=76$, a contradiction.
If we assume that there exist a 9-plane then a 2-line in this 9-plane is forced to be contained in an 8-plane, which was already ruled out.
In the sme wy we can rule out the existence of planes of multiplicity $-1,0\pmod{3}$.

Next assume that $\mathcal{F}\vert_S$ is a $(34,10)$ minihyper. It is immediate that there exist no 11-planes in $S$ since there exist no $(11,3)$-minihypers in $\PG(2,3)$.
From this point on the proof completed in the case of $25$-solids.
\end{proof}

\begin{lemma}
	\label{lma:40}
	Let $\mathcal{F}$ be a $(70,22)$-minihyper in $\PG(4,3)$, and let $S$ be a solid of multiplicity $40$. Then $\mathcal{F}\vert_S$ is a $(40,12)$-minihyper
	in $\PG(3,3)$, that is reducible to a $(39,12)$-minihyper.
\end{lemma}

\begin{proof}
We have to rule out the existence of planes of multiplicity $\equiv-1\pmod{3}$. By the previous results such a plane can be contained only in 40-solids. If we denote its multiplicity by
$x$ we have $x\equiv-1\pmod{3}$ and $4\cdot40-3x=70$, But the equation  implies $x\equiv0\pmod{3}$, a contradiction.
\end{proof}

Now it is easily checked that a $(70,22)$-minihyper in $\PG(4,3)$ with hyperplanes of multiplicity at most 43 satisfies the conditions of Theorem ~\ref{thm:70-22}:

- Condition (1) follows by Lemmas~\ref{lma:divisible}, and \ref{lma:28-46};

- Condition (2) follows by Lemmas~\ref{lma:22-6}, \ref{lma:30-9}, and \ref{lma:40};
	
- Condition (3) follows by Lemmas~\ref{lma:43}, \ref{lma:22irreducible}, and \ref{lma:25-34}.

\section*{Acknowledgments}  
The first author  was  supported    by  the  
Bulgarian National Research  Fund  under Contract
KP-06-N72/6-2023.
The research of the second author was supported by
Sofia  University  under  Contract  80-10-72/25.04.2023.
The second author  was  supported  by  the
Research  Fund  of  Sofia  University  under  Contract
80-10-164/18.04.2024.
The research of the third author was supported by the NSP P. Beron project CP-MACT.

\end{document}